\documentclass[11pt]{scrartcl}

\usepackage[english]{babel}
\usepackage{amsmath,amssymb,amsthm}
\usepackage{enumerate}
\usepackage{mathrsfs}

\newtheorem{lemma}{Lemma}
\newtheorem*{theorem}{Theorem}

\title{An Inductive Proof of Whitney's \\ Broken Circuit Theorem}
\author{Klaus Dohmen\\ Hochschule Mittweida\\ Technikumplatz 17\\ 09648 Mittweida, Germany}

\begin{document}

\maketitle

\begin{abstract}
We present a new proof of Whitney's broken circuit theorem based on induction on the number of edges 
and the deletion-contraction formula.
\end{abstract}

\section{Introduction}

For any finite graph $G$, 
the chromatic polynomial of $G$ is a polynomial in $\mathbb{Z}[\lambda]$
which for any $\lambda\in\mathbb{N}$ evaluates to the number of vertex colorings of $G$
with at most $\lambda$ colors such that
adjacent vertices receive different colors.

This polynomial was introduced by Birkhoff \cite{Birkhoff} in 1912,
who showed that it is a monic polynomial whose degree coincides with the number of vertices of $G$,
provided $G$ is loop-free, and whose coefficients alternate in sign. 
We use
\begin{equation}
\label{darstellung}
P_G(\lambda) = \sum_{k=0}^{n(G)} (-1)^k a_k (G) \lambda^{n(G)-k}
\end{equation}
to denote the chromatic polynomial of $G$ and its coefficients, respectively,
where $n(G)$ denotes the number of vertices of $G$.
If $G$ contains loops, then $P_G(\lambda) = 0$.

Whitney's broken circuit theorem \cite{Whitney} gives a combinatorial interpretation of the coefficients in terms
of so-called broken circuits. 
A \emph{broken circuit} of a graph $G$ arises from the edge-set of a cycle of $G$
by removing its maximum edge with respect to some fixed linear ordering relation on the edges. 

\begin{theorem}[Whitney, 1932]
Let $G$ be a finite graph with a linear ordering relation on its edges.
Then, for $k=0,\dots,n(G)$ the coefficient $a_k(G)$ equals the number of $k$-subsets of the edge-set of $G$
not including any broken circuit of $G$ as a subset.
\end{theorem} 

Whitney's original proof \cite{Whitney} uses the inclusion-exclusion principle,
followed by some combinatorial arguments;
for a detailed account, we refer to the textbook of Biggs \cite[Ch.~10]{Biggs}.
There are two more recent alternative ways to prove Whitney's result:
\begin{itemize}
 \item using an explicit bijection due to Blass and Sagan \cite{Blass-Sagan}, or
 \item using recent variants of the inclusion-exclusion principle \cite{Dohmen}.
\end{itemize}
Although many famous results on the chromatic polynomial are proved by induction,
so far no such simple proof has been given for Whitney's broken circuit theorem.

In this paper, we provide a proof 
using induction on the number of edges and the deletion-contraction formula \cite{Read},
which states that for any graph $G$ and any edge $e$ of $G$,
\begin{align}
\label{decomposition1}
P_G(\lambda) & = P_{G- e}(\lambda) - P_{G|_e}(\lambda), 
\end{align}
where $G- e$ resp\@. $G|_e$ are obtained from $G$ by removing
resp\@. contracting $e$.
The main idea is to show that if the statement of Whitney's broken circuit theorem holds for both $G- e$ and $G|_e$,
then it also holds for $G$. 

Due to the presence of a linear ordering relation on the edges,
we need a precise definition of the edge contraction operation such that $G|_e$ inherits the linear ordering relation from $G$ in such a way 
that the induction step becomes feasible. 
This definition is given in Section~\ref{prelim}.
The proof of Whitney's theorem follows in Section~\ref{proofsec}. 

\section{Preliminaries}
\label{prelim}

Throughout the paper, a \emph{graph} is a triple $G=(V,E,\varphi)$,
where $V$ and $E$ are sets and $\varphi$ is a mapping from $E$ to the set of one- or two-element subsets of $V$. 
The elements of $V$ and $E$ are called \emph{vertices} and \emph{edges}, respectively;
$\varphi$ is called an \emph{incidence mapping}. 

A graph is \emph{finite} if both its vertex-set and edge-set are finite.
We use $V(G)$ and $E(G)$ to denote the vertex-set and edge-set of $G=(V,E,\varphi)$, respectively, 
and $n(G)$ resp\@. $m(G)$ to denote their cardinalities.

If $v\in\varphi(e)$ for some $v\in V(G)$ and $e\in E(G)$, we say that $v$ is \emph{incident} with $e$ in $G$.
Two vertices $v,w\in V(G)$ are \emph{adjacent} in $G$ if $\varphi(e) = \{v,w\}$ for some edge $e\in E(G)$.

A \emph{loop} is an edge $e$ such that $|\varphi(e)|=1$. 
Edges $e$ and $e'$ are called \emph{parallel} if $e\neq e'$ and $\varphi(e) = \varphi(e')$.
A graph without loops and parallel edges is called \emph{simple}. 


Any edge $e\in E(G)$ induces an equivalence relation $\sim_e$ on $V(G)$ via
\begin{align*}
 v\sim_e w & \quad :\Longleftrightarrow \quad \text{$v=w$~~or~~$\varphi(e)=\{v,w\}$},\\
\intertext{and an equivalence relation $\sim_e$ on $E\setminus\{e\}$ (using the same symbol) via}
x\sim_e y & \quad :\Longleftrightarrow \quad \text{$x=y$~~or~~$\varphi(e) = \varphi(x) \bigtriangleup \varphi(y)$}
\end{align*}
where $\bigtriangleup$ denotes symmetric difference of sets. We use $[x]_e$ to denote the equivalence class of $x$ with respect to $\sim_e$,
regardless of whether $x$ is a vertex or an edge. 

\emph{Deletion} and \emph{contraction} of an edge $e$ from $G=(V,E,\varphi)$ are defined by
\begin{align*}
  G - e & \,:=\, \left( V(G),\, E(G)\setminus\{e\},\, \varphi|_{E(G)\setminus\{e\}}\right),\\
  G|_e & \,:=\, \left( V(G)/\!\sim_e, \,\left.\left( E(G)\setminus\{e\}\right)\right/\!\sim_e,\,\varphi_e \right),
\end{align*}
where for any $x\in E(G)\setminus\{e\}$,
\[ \varphi_e \left( [x]_e \right) := \left\{ [v]_e\mathrel| v\in\varphi(x)\right\}. \]
For any such $x$, we identify $[x]_e$ with the \emph{maximum} edge $y\in [x]_e$, for some given linear ordering relation on the edges of $G$.
In this way, $E(G|_e)$ becomes a subset of $E(G)$ and inherits the linear ordering relation.

With the above definitions, the class of simple graphs is closed under deletion and contraction.

\section{Proof of Whitney's broken circuit theorem}
\label{proofsec}

Throughout, $G$ is a graph whose edge-set is endowed with a linear ordering relation. 
Apart from Lemma \ref{lemma0}, $G$ is
required to be loop-free.

\begin{lemma}
\label{lemma0}
Let $e$, $f$ be parallel edges where $e<f$, and $X\subseteq E(G)$.
Then, $X$ includes no broken circuit of $G$ if and only if 
$X\subseteq E(G-e)$ and $X$ includes no broken circuit of $G- e$.
\end{lemma}

\begin{proof}
The statement holds since, by the assumptions, $\{e\}$ is a broken circuit of $G$. 
\end{proof}

\begin{lemma}
\label{lemma1}
Let $e$ be the minimum edge of $G$, and $Y\subseteq E(G- e)$. 
Then, $Y$ includes a broken circuit of $G- e$ 
if and only if $Y$ includes a broken circuit of $G$.
\end{lemma}

\begin{proof} 
The first direction is obvious. For the opposite direction,
assume $C\subseteq Y$ for some broken circuit $C$ of $G$.
By definition of a broken circuit,
there is some $f\in E(G)$ such that $f>\max C$ and $C\cup\{f\}$ is the edge-set of a cycle of $G$.
From $C\subseteq Y$ and $Y\subseteq E(G- e)$ we conclude that $e\notin C$. 
Since $e$ is the minimum edge of $G$, $f\neq e$ (otherwise $e>\max C$).
Therefore, $C\cup\{f\}$ is the edge-set of a cycle of $G- e$ and hence,
$C$ is a broken circuit of $G- e$. 
Consequently, $Y$ includes a broken circuit of $G- e$.
\end{proof}

\begin{lemma}
\label{lemma2}
Let $e$ be the minimum edge of $G$, and $X\subseteq E(G)$ satisfying $e\in X$.
If $X$ includes no broken circuit of $G$,
then 
\begin{enumerate}[(a)]
\item $X\setminus\{e\} \subseteq E(G|_e)$;
\item $X\setminus\{e\}$ includes no broken circuit of $G|_e$.
\end{enumerate}
\end{lemma}

\begin{proof}
(a) On the contrary, assume $X\setminus\{e\} \not\subseteq E(G|_e)$. Then, $f\not\in E(G|_e)$ for some $f\in X\setminus\{e\}$.
By definition of $G|_e$, $[f]_e\in E(G|_e)$. Since $[f]_e$ is identified with the maximum edge $h\in [f]_e$, it follows that $h\in E(G|_e)$
for some $h\in E(G)\setminus\{e\}$ where $h\sim_e f$ and $h>f$. Therefore, $\{e,f,h\}$ is the edge-set of a  triangle in $G$, 
and $h = \max\{e,f,h\}$. Hence, $\{e,f\}$ is a broken circuit of $G$, and this broken circuit is a subset of $X$.

(b) On the contrary, assume $C\subseteq X\setminus\{e\}$ for some broken circuit $C$ of $G|_e$.
Then, there exists $f\in E(G|_e)$ such that $f>\max C$ and
$C\cup\{f\}$ is the edge-set of a cycle of $G|_e$.
Hence, $C\cup\{f\}$ or $C\cup\{e,f\}$ is the edge-set of a cycle of $G$.
Since $f>\max C$ and $f>e$, this implies that $C$ or $C\cup\{e\}$ is a broken circuit of $G$. 
Since $C\cup\{e\}\subseteq X$ (because $C\subseteq X\setminus\{e\}$ and $e\in X$), it follows that
$X$ includes a broken circuit of $G$.
\end{proof}

\begin{lemma}
\label{lemma3}
Let $e$ be the minimum edge of $G$, and $Y\subseteq E(G|_e)$. 
If $Y$ includes no broken circuit of $G|_e$, 
then $Y\cup\{e\}$ includes no broken circuit of $G$.
\end{lemma}

\begin{proof}
On the contrary, assume that $C\subseteq Y\cup\{e\}$ for some broken circuit $C$ of $G$.
Then, there exists $f\in E(G)$ such that $f>\max C$ and $C\cup\{f\}$ is the edge-set of a cycle of $G$.
We distinguish two cases:

  \textit{Case 1:} 
  $e\notin C$.
  Then, $C\subseteq Y$ and hence, $C\subseteq E(G|_e)$. 
  By this, and since $C\cup\{f\}$ is the edge-set of a cycle of $G$,
  $C\cup\{h\}$ is the edge-set of a cycle of $G|_e$ for some $h\in [f]_e$ satisfying $h\ge f$.
  By transitivity, $h > \max C$. Hence, $C$ is a broken circuit of $G|_e$.
  
  \textit{Case 2:} 
  $e\in C$.
  Then, $C\setminus \{e\} \subseteq Y$. We observe that $C\cup\{f\}$ cannot be the edge-set of a triangle in $G$.
  Otherwise, $C\cup\{f\} = \{e,f,h\}$ for some $h\sim_e f$.
  On the one hand, $f>\max C = h$; since $h\sim_e f$, $h\notin E(G|_e)$.
  On the other hand, $h\in C\setminus\{e\} \subseteq Y \subseteq E(G|_e)$\,---\,a contradiction.
  Therefore, $C\cup\{f\}$ is the edge-set of a cycle of $G$ of length at least four.
  Contracting $e$ gives a cycle in $G|_e$ with edge-set $(C\setminus\{e\}) \cup \{f\}$.
  Since $f>\max C$ and $f>e$ it follows that $C\setminus\{e\}$ is a broken circuit of $G|_e$,
  which is included by $Y$.
  
In both cases, it is shown that $Y$ includes a broken circuit of $G|_e$.
\end{proof}

\begin{lemma}
\label{lemma4}
Let $e$ be the minimum edge of $G$, and $k\in\mathbb{N}$. 
\begin{enumerate}[(a)]
\item There are as many $k$-subsets of $E(G- e)$ including no broken circuit of $G- e$
as there are $k$-subsets of $E(G)$ not containing $e$ and not including any broken circuit of $G$.
\item There are as many $(k-1)$-subsets of $E(G|_e)$ including no broken circuit of $G|_e$
as there are $k$-subsets of $E(G)$ containing $e$ and not including any broken circuit of $G$.
\end{enumerate}
\end{lemma}

\begin{proof}
\begin{enumerate}[(a)]
  \item This is an immediate consequence of Lemma~\ref{lemma1}.
  \item Let $\mathscr{A}_{k-1}(G|_e)$ resp\@. $\mathscr{A}_k^{\prime}(G)$ denote the set of subsets of the first resp\@. second type.
  Consider the mapping $\beta : \mathscr{A}_k^{\prime}(G) \rightarrow \mathscr{A}_{k-1}(G|_e)$, $X\mapsto X\setminus\{e\}$.
  In fact, Lemma~\ref{lemma2} guarantees that $\beta(\mathscr{A}_k^{\prime}(G))\subseteq \mathscr{A}_{k-1}(G|_e)$. 
  Since $\beta$ is one-to-one (trivially) and onto by Lemma~\ref{lemma3}, it is a one-to-one correspondence between $\mathscr{A}_k^{\prime}(G)$ and $\mathscr{A}_{k-1}(G|_e)$. \qedhere
\end{enumerate}
\end{proof}

Using the preceding results, we now prove Whitney's broken circuit theorem.

\begin{proof}
By Lemma \ref{lemma0} we may assume that $G$ has no parallel edges.
If $G$ contains loops, then $\emptyset$ is a broken circuit of $G$, and $P_G(\lambda)=0$.
It is easy to verify that the statement of the theorem holds in this particular case.

In the remainder of the proof, we may assume that $G$ is a simple graph.
We proceed by induction on the number of edges.

If $m(G)=0$, the statement is obvious.
The same applies if $m(G)>0$ and $k=0$.

For the induction step, let $e$ be the minimum edge of $G$.
By (\ref{darstellung}) and (\ref{decomposition1}) we have
\begin{align}
\label{decomposition2}
a_k(G) & = a_k(G- e) + a_{k-1}(G|_e) \quad (k=1,\dots,n(G)).
\end{align}
Since the graphs $G- e$ and $G|_e$ are both simple and have fewer edges than $G$,
we may apply the induction hypothesis for both $G- e$ and $G|_e$. Thus, we find that
\begin{align*}
a_k(G- e) & = \text{number of $k$-subsets of $E(G- e)$ including no broken circuit of $G- e$}, \\
a_{k-1}(G|_e) & = \text{number of $(k-1)$-subsets of $E(G|_e)$ including no broken circuit of $G|_e$}.
\end{align*}
Therefore, by Lemma~\ref{lemma4} the theorem is proved.
\end{proof}

\end{document}